  \documentclass[11pt]{amsart}

\usepackage{amsmath,amsthm,amsfonts,latexsym,amscd,amssymb,enumerate,color,hyperref}
\usepackage{pdftricks}
\input{xypic}

\usepackage{graphicx}
\begin{psinputs}
 \usepackage[usenames,dvipsnames]{pstricks}
 \usepackage{epsfig}
  \usepackage{pst-grad} % For gradients
 \usepackage{pst-plot} % For axes
\end{psinputs}

\usepackage{epstopdf}

\def \<{\langle}
\def \>{\rangle}

\newcommand{\R}{\mathbb{R}}

\newcommand{\calR}{\mathcal{R}}

\newcommand{\tf}{\widetilde{d}}

\newcommand{\eps}{\varepsilon}

\newcommand{\la}{\langle}
\newcommand{\ra}{\rangle}

\renewcommand{\prod}[2]{\la\,#1 , #2\,\ra}
\newcommand{\ip}[2]{\ensuremath{\langle #1,#2\rangle}}

\newcommand{\dfq}{\partial_q\tf_\rho}
\newcommand{\dfy}{\partial_y\tf_\rho}
\newcommand{\nb}{\nu(M^{p_0})}

\DeclareMathOperator{\dist}{dist}

\newtheorem{thm}{Theorem}[section]
\newtheorem{cor}[thm]{Corollary}
\newtheorem{lem}[thm]{Lemma}

\newtheorem{theorem}{Theorem}

\theoremstyle{definition}

\newtheorem{rem}[thm]{Remark}

\newtheorem*{ack}{Acknowledgements}

\numberwithin{equation}{section}

\begin{document}

% AUTHOR 1

\author[Galaz-Garc\'ia]{Fernando Galaz-Garc{\'\i}a$^*$}
\address[Galaz-Garc\'ia]{Institut f\"ur Algebra und Geometrie, Karlsruher Institut f\"ur Technologie (KIT), Germany.}
\email{galazgarcia@kit.edu}
\thanks{$^{*}$Supported in part by SFB 878: \emph{Groups, Geometry and Actions} at the University of M\"unster.}

% AUTHOR 2

\author[Guijarro]{Luis Guijarro$^{**}$}
\address[Guijarro]{Department of Mathematics, Universidad Aut\'onoma de Madrid, and ICMAT CSIC-UAM-UCM-UC3M, Spain.}
\email{luis.guijarro@uam.es}
\thanks{$^{**}$Supported by research grants  MTM2011-22612 from the Ministerio de Ciencia e Innovaci\'on (MCINN) and MINECO: ICMAT Severo Ochoa project SEV-2011-0087.}

\title{Every point in a Riemannian manifold is critical}
\date{\today}

% MATH SUBJECT CLASSIFICATION AND KEYWORDS

\subjclass[2010]{53C20,53C22} 
\keywords{distance function, critical point, geodesic loop}

% ABSTRACT

\begin{abstract} 
We show that for any point $p$ in a closed Riemannian manifold $M$, there exists at least one point $q\in M$ such that $p$ is critical for the distance function from $q$. We also show that such a point $q$  cannot always be reached with geodesic loops based at $q$ with midpoint $p$. 
\end{abstract}

\maketitle

% SECTION: INTRO/MAIN RESULTS
%------------------------------------------------

\section{Main results}
Critical point theory has been of central importance in many areas of mathematics. In Riemannian geometry, however, the most natural functions are distance functions and, due to their possible lack of differentiability at the cut locus, it was not clear for some time what a critical point should be. This situation was corrected in \cite{GroShio}, where a point $p$ in a Riemannian manifold $M$ was considered to be \emph{critical} for $q$ if the set of tangent vectors  to minimal geodesics connecting $q$ to $p$ forms a $\pi/2$-net  at the unit tangent sphere at $p$.  

The surveys \cite{Chee} and \cite{Gro}, as well as \cite[Chapter 11]{Pet}, make plain the importance of critical point theory for distance functions and  its use in giving unified conceptual proofs of many of the main results in Riemannian geometry. Nevertheless, it is rare to see critical points studied by themselves. An exception to this is \cite{BIVZ}, where the authors prove that any point $p$  in a closed Alexandrov surface has to be critical for some point $q$. The purpose of this note is to give a short proof of this fact  for closed Riemannian manifolds of arbitrary dimension. Our proof, however, does not carry over to Alexandrov spaces.

% THM

\begin{theorem}
\label{T:Smooth}
Let $M$ be a closed Riemannian manifold. Then, for any point $p\in M$, there exists at least one point $q\in M$ such that $p$ is critical for the distance function from $q$.
\end{theorem}
Observe that this theorem is no longer true in the non-compact case. For example, in Euclidean space no point is critical for any point.  A naive approach  to proving Theorem~\ref{T:Smooth} would be to look for the point $q$ among the points at maximal distance from $p$. This, however, fails, as  shown for the surface of revolution in Figure~\ref{F:Ellipsoid}. The point $q$ farthest away from $p$  is critical for the distance function from $p$ but, clearly, $q$ is not critical for $p$, as pointed out to us by J.~Itoh.

% FIGURE 1
%-----------------

\begin{figure}
\label{F:Ellipsoid}
\begingroup%
  \makeatletter%
  \providecommand\color[2][]{%
    \errmessage{(Inkscape) Color is used for the text in Inkscape, but the package 'color.sty' is not loaded}%
    \renewcommand\color[2][]{}%
  }%
  \providecommand\transparent[1]{%
    \errmessage{(Inkscape) Transparency is used (non-zero) for the text in Inkscape, but the package 'transparent.sty' is not loaded}%
    \renewcommand\transparent[1]{}%
  }%
  \providecommand\rotatebox[2]{#2}%
  \ifx\svgwidth\undefined%
    \setlength{\unitlength}{276.46303689bp}%
    \ifx\svgscale\undefined%
      \relax%
    \else%
      \setlength{\unitlength}{\unitlength * \real{\svgscale}}%
    \fi%
  \else%
    \setlength{\unitlength}{\svgwidth}%
  \fi%
  \global\let\svgwidth\undefined%
  \global\let\svgscale\undefined%
  \makeatother%
  \begin{picture}(1,0.47504526)%
    \put(0,0){\includegraphics[width=\unitlength]{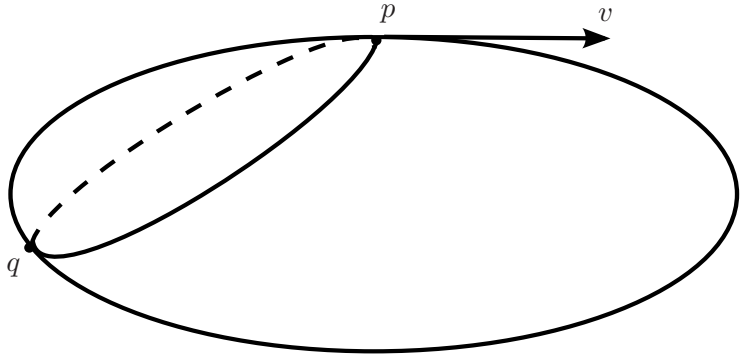}}%
    \put(0.50906515,0.45521959){\makebox(0,0)[lb]{$p$}}%
    \put(-0.00091796,0.10909302){\makebox(0,0)[lb]{$q$}}%
    \put(0.8054782,0.45706107){\makebox(0,0)[lb]{$v$}}%
  \end{picture}%
\endgroup%
\caption{The point $q$ is critical for $p$ but not viceversa: the tangent vector $v$ forms an angle greater than $\pi/2$ with the two geodesics connecting $q$ to $p$.}
\end{figure}

%-------------------

Once we know that any point $p\in M$ is critical for some $q$, a natural question arises: \emph{What type of critical point is $p$?} Among critical points, the simplest situation corresponds to midpoints of geodesic loops from $q$. For this case, there are exactly two minimal geodesics from $q$ to $p$ forming an angle $\pi$ at $p$. The criticality condition is clearly satisfied, since any point in the tangent sphere at $p$ is contained in one of the two closed hemispheres that the two geodesics determine at $p$. It is then interesting to know if this is always the case, namely: \emph{Is it true that for any Riemannian metric on $M$, every point is the midpoint of a geodesic loop with the property that half of the loop is minimal?} We thank Burkhard Wilking for bringing this question to our attention, and we give a negative answer to it.  

% THM

\begin{theorem}
\label{T:Loops} Let $M$ be a closed smooth manifold. Then there exists a Riemannian metric on $M$ such that not every point is the midpoint of a geodesic loop with the property that half the loop is minimal.
\end{theorem}

We prove  Theorem~\ref{T:Smooth} in Section~\ref{S:Proof_Smooth}. We prove Theorem~\ref{T:Loops} in Section~\ref{S:Loops}, using a construction introduced by Gluck and Singer \cite{GS} to construct Riemannian manifolds with prescribed cut locus.

% ACKNOWLEDGEMENTS

\begin{ack}The authors would like to thank Burkhard Wilking, for a critical observation, the Posgrado de Excelencia Internacional en Matem\'aticas at the Universidad Aut\'o\-noma de Madrid, where the work contained in this article was initiated, and the Mathematisches Forschungsinstitut Oberwolfach, for its hospitality and its pleasant work environment. Luis Guijarro would also like to thank Jin-ichi Itoh, for bringing to his attention the work in \cite{BIVZ}.
\end{ack}

% SECTION: PROOF - SMOOTH
%---------------------------------------------

\section{Proof of Theorem~\ref{T:Smooth}}
\label{S:Proof_Smooth}

Let $M$ be a closed Riemannian manifold. We will prove the theorem by contradiction. Thus,  we will assume that there exists some point $p_0$ in $M$ that is not critical for any other point in $M$.  Consider the submanifold $M^{p_0}=M \times \{p_0\} $ inside the product $M\times M$. We will  construct a section of the normal bundle $\nb$ of $M^{p_0}$ with an isolated zero of non-zero index. Since the normal bundle is trivial, we derive a contradiction by looking at its Euler class (cf.~\cite{BT}).

We would like to define the section of $\nb$ at a point $(q,p_0)\in M^{p_0}$ as the gradient vector at $p_0$ of the distance function $d_q$ from $q$. To do so, we will use the non-criticality of $p_0$ with respect to $q$. Because distance functions are not everywhere differentiable, we need to smooth out the function $d_q$ first. After doing this, we must modify the section so that it depends smoothly on $q$. Therefore, we need to smooth out the distance function on $M$ with respect to each one of its two variables. To achieve this, we apply the smoothing technique for distance functions as appearing, for instance, in \cite{GroShio}. We divide the proof in seven steps.
\\

% STEP 1
%----------------------

\noindent\textbf{Step 1.} Choose some function $\phi:[0,\infty)\to [0,1]$ with support contained in $[0,1)$, $\phi'\leq 0$, and $\phi\equiv 1$ 
at points close to $0$. 
\\

% STEP 2
%----------------------

\noindent\textbf{Step 2.} Let $i(M)$ be the injectivity radius of $M$ and fix $p\in M$. Given $\rho < i(M)$,  denote by $d\mu_p$ the measure on  $B_\rho(p)$ induced by $\exp_p$ and the Lebesgue measure on $B_\rho(0)$.
\\

% STEP 3
%----------------------

\noindent\textbf{Step 3.}  Define a $\rho$-mollifier kernel as the map $\Phi_\rho:M\times M\to \R$, given by

\[
 \Phi_\rho(x,y)=\left(\int_{B_\rho(x)}\,\phi\left(\frac{1}{\rho} d(x,\cdot)\right)\,d\mu_x\right)^{-1}\cdot 
\phi\left(\frac{1}{\rho}d(x,y) \right).
\]

It is clear that $\Phi_{\rho}$ is smooth, symmetric in $x$ and $y$, and that the first factor in the above expression is independent of the $x$ used. 
\\
   
% STEP 4
%----------------------

\noindent\textbf{Step 4.}  By \cite[Proposition~2.1]{GroShio}, the function
\[
d_{q,\rho}(y):=\int_{B_\rho(y)}\,d_q(z)\, \Phi_\rho(y,z)\,d\mu_y(z)
\] 
is smooth with respect to $y$. We need to improve this to obtain smoothness with respect to $q$. 

We now consider the point $q$ as a variable of $d_{q,\rho}(y)$. After a second smoothing, we obtain the function $\tf_\rho:M\times M\to\R$ given by
\[
\tf_\rho(q,y):=\int_{p\in B_\rho(q)}\,d_{p,\rho}(y)\,\Phi_\rho(q,p)\,d\mu_q(p).
\]
This expression can be expanded into 

\begin{equation}
\label{Eq:Dbl_smooth}
\tf_\rho(q,y)=\int_{p\in B_\rho(q)}\,\int_{z\in B_\rho(y)}\,d(p,z)\Phi_{\rho}(y,z)\Phi_\rho(q,p)\, d\mu_y(z)\,d\mu_q(p),
\end{equation}

\noindent where, to facilitate the reading, we have incorporated into each measure the variable with respect to which we integrate. Equation~\eqref{Eq:Dbl_smooth} above can be interpreted as the smoothing of the function $d:M\times M\to \R $ after convolution with a product mollifier kernel. Therefore, the function $\tf_{\rho}$ is smooth.
\\

 % STEP 5: CONSTRUCTION OF THE GENERALIZED GRADIENT AWAY FROM CRITICAL POINTS
%----------------------

\noindent\textbf{Step 5.}  
Let $u\in T_pM$. Following \cite[p.~208]{GroShio}, we construct a smooth vector field $U$ on $B_{\rho}(p)$ as follows. Let $\gamma$ be the unique geodesic with $\gamma'(0)=u$ and define for each $y\in B_\rho(p)$ a smooth curve $\gamma_y$ by 
\[
\gamma_y(t)=\exp_{\gamma(t)}\big(P_{\gamma(t)}\big(\exp^{-1}_{\gamma(0)}(y)\big)\big).
\]  
We then set $U_y=\gamma'(0)$.

Recall that the distance function $d:M\times M\to \R$ is smooth almost everywhere with bounded gradient. The remark before \cite[Theorem~2.3]{GroShio} implies the following two lemmas:
 
 % LEM

\begin{lem}
Let $u_1\in T_qM$, $u_2\in T_xM$ and let $U_1$, $U_2$ be the vector fields constructed on $B_\rho(q)$, $B_\rho(y)$, respectively, from $u_1$ and $u_2$ as indicated above.Then 
\begin{multline}
\ip{\nabla\tf_\rho(q,y)}{(u_1,u_2)}= \\
=\int_{M\times M}\ip{\nabla d}{(U_1,U_2)}(p,z)
\Phi_{\rho}(y,z)\Phi_\rho(q,p)\, d\mu_y(z)\,d\mu_q(p).
\end{multline}
\end{lem}

% LEM

\begin{lem}
\label{L:close}
Let $(q,y)\in M\times M$.
For any $\eps>0$, there exists $\rho<i(m)$ such that 
\begin{align}
\left\|(\nabla\tf_\rho)_{(q,y)} - \int_{B_\rho(q)}\int_{B_{\rho}(y)} (P_{q,y}\circ\nabla d)\cdot \Phi_\rho(q,\cdot)\Phi_\rho(y,\cdot)d\mu_yd\mu_q\right\|<\varepsilon.
\end{align}
\end{lem}
In Lemma~\ref{L:close} above, $P_{q,y}$ denotes parallel transport along the unique geodesic from $(p,z)$ to $(q,y)$, for $(p,z)$ in the domain of integration. 
\\

% STEP 6
%----------------------

\noindent\textbf{Step 6.}  Use  $\dfq(q,y)$ and $\dfy(q,y)$, respectively, to denote the components of the gradient  $(\nabla\tf_\rho)_{(q,y)}$ under the splitting 
\[
T_{(q,y)}(M\times M) = T_qM\oplus T_yM.
\]

%--------------------------------------------------

% LEM

\begin{lem}
If $p_0$ is not critical for any point $q$ different from $p_0\in M$, then there is $\rho >0$ sufficiently small such that 
$\dfy(q,p_0)\neq 0$ for any point $q\neq p_0$.
\end{lem}

\begin{proof}
Let $\eps>0$.
Using Lemma~\ref{L:close} at the point $(q,p_0)\in M\times M$, we obtain  
\begin{align}
\label{Eq:Norm}
\left\| (\nabla\tf_\rho)_{(q,p_0)} - \int_{B_\rho(q)}\int_{B_{\rho}(p_0)} (P_{q,p_0}\circ\nabla d)\cdot \Phi_\rho(q,\cdot)\Phi_\rho(p_0,\cdot)d\mu_{p_0}d\mu_q\right\|<\varepsilon.
\end{align}
for some $\rho< i(M)$ small enough.
Now we look at the second component of the vector inside the integral in Equation~\eqref{Eq:Norm}. We may write
\[
\nabla d_{(p,z)}= ((\nabla d_z)(p),(\nabla d_p)(z)),
\]
where $d_z,d_p :M\to \R$ are, respectively, the distance functions from $z$ and $p$. 
 Using that parallel transport in the product $M\times M$ is the product of parallel transport on each factor, we deduce that 
 \[
\left\| \dfy(q,p_0)- \int_{p\in B_\rho(q)}\int_{B_{\rho}(p_0)} (P_{q,p_0}\circ\nabla d_p)\cdot \Phi_\rho(q,p)\Phi_\rho(p_0,\cdot)d\mu_{p_0}d\mu_q\right\|<\eps.
 \]
 
Consider some $q\in M$; since $p_0$ is not critical for $q$, there are balls of radius $\rho$ around $q$ and $p_0$ such that no point in $B_\rho(p_0)$ is critical for any point in $B_\rho(q)$. For any pair of different points $p$, $y$, denote the set of unit vectors at $y$ tangent  to geodesics connecting $p$ to $y$ as $\calR_{p,y}$. Non-criticality implies that for any $p\in B_\rho(q)$, $y\in B_\rho(p_0)$, the set $\calR_{p,y}$ is contained in an open half-space of $T_{y}M$.  As in 
\cite{GroShio} construct a non-zero vector $X(p,y)\in T_{y}M$ forming an angle greater than $\pi/2+\delta$ with every vector in $\calR_{p,y}$ for some $\delta>0$ independent of $p$ and $y$ in the above balls.  Furthermore, the norm of $X$ is bounded below.

This implies that there is some $\tau<0$ such that  $\ip{X}{\nabla d_p}\leq \tau$ for every pair  $(p,y)\in B_\rho(q)\times B_\rho(p_0)$. If $\rho$ is small enough, we can assure that 
$\ip{X}{P_{q,p_0}\nabla d_p}\leq \tau/2$, and as a consequence 
$\dfy(q,p_0)\neq 0$ as we wanted to show.
\end{proof}

% STEP 7
%---

\noindent\textbf{Step 7.} Identify the normal bundle $\nu(M^{p_0})$ with $M^{p_{0}}\times (\{0\}\oplus T_{p_{0}}M)$ via the decomposition  $T(M\times M) = TM \oplus TM$. Consider the section $V(q,p_0)$ of $\nu(M^{p_0})$ that takes the value $(0,\dfy(q,p_0))$ at the point $(q,p_0)$. Observe that $V$ is defined on $M^{p_0}\setminus(p_0, p_0)$, and does not vanish. On a sufficiently small geodesic ball centered at $p_0$, the vector $\dfy(q,p_0)$ is arbitrarily close, as $\rho$ approaches $0$,  to the tangent vector at $p_0$ to the unique geodesic connecting $q$ to $p_0$. 
This, and a simple partition of unity argument, permits us to refine the second component of the section $V$ to agree with the gradient field of the function $\dist_{p_0}^2$
It follows that the index of $V$ at $(p_0,p_0)$  is non-zero (cf.~\cite[Theorem 11.16]{BT}). This is a contradiction, since the normal bundle $\nu(M^{p_0})$ is trivial. 

\hfill $\square$

% REMARK

\begin{rem} A similar statement to Theorem~\ref{T:Smooth} is true if we replace $q$ by a set of $k\geq 2$ distinct points in $M$. Indeed, it is not difficult to find $k\geq 2$ distinct points $q_1,\ldots,q_k$ in $M$ such that $p$ is critical for the distance function to the union of the $q_i$. To find such a collection of points, take a sufficiently small geodesic sphere $S(\delta)$ centered at $p$, with $\delta<\mathrm{inj}(M)$, and let $q_1$ and $q_2$ be antipodal points in $S(\delta)$. We then complete the collection by arbitrarily adding $k-2$ distinct points in $S(\delta)$.
\end{rem}

% REMARK

\begin{rem} The argument used in the proof of Theorem~\ref{T:Smooth} yields the following:

% THM

\begin{thm}Let $M$ be a closed Riemannian manifold and let $f:M\times M\rightarrow [0,\infty)$ be a smooth function such that $f(p,q)=0$ if and only if $p=q$. Then, for any point $p\in M$, there exists at least one point $q\in M$ such that $p$ is critical for $f(\cdot,q)$.
\end{thm}

Pushing farther the argument, one could extend the preceding theorem to the case where $f$ is a Lipschitz function with continuous non-zero gradient close to the diagonal (although not necessarily defined on it).
\end{rem}

% SECTION: FURTHER COMMENTS
%----------------------------------------------------

\section{Criticality through geodesic loops}
\label{S:Loops}

\begin{proof}[Proof of Theorem~\ref{T:Loops}]
First, observe that if $p$ is critical for $q$, then $q$ is in the cut locus of $p$, that we denote as $C(p)$.  In our case, if $p$ is the midpoint of a geodesic loop then there should be a point $q$ in $C(p)$ and two geodesic segments from $q$ to $p$ whose tangent vectors are antipodal at $p$. Our strategy would be to construct a metric on $M$ such that this condition is not satisfied for any point in $C(p)$.

We will start by assuming  $M$ to be diffeomorphic to the standard sphere $\mathbb{S}^n$; the desired metric will be obtained using a 
construction of Gluck and Singer (see \cite{GS}). They show how to glue geodesic fields in the northern and southern hemispheres along a preasigned diffeomorphism of the equator $\mathbb{S}^{n-1}\subset\mathbb{S}^{n}$  so that in the  metric the glued curves remain geodesics.

% FIGURE 2
%--------------------

\begin{figure}
\begingroup%
  \makeatletter%
  \providecommand\color[2][]{%
    \errmessage{(Inkscape) Color is used for the text in Inkscape, but the package 'color.sty' is not loaded}%
    \renewcommand\color[2][]{}%
  }%
  \providecommand\transparent[1]{%
    \errmessage{(Inkscape) Transparency is used (non-zero) for the text in Inkscape, but the package 'transparent.sty' is not loaded}%
    \renewcommand\transparent[1]{}%
  }%
  \providecommand\rotatebox[2]{#2}%
  \ifx\svgwidth\undefined%
    \setlength{\unitlength}{358.30567673bp}%
    \ifx\svgscale\undefined%
      \relax%
    \else%
      \setlength{\unitlength}{\unitlength * \real{\svgscale}}%
    \fi%
  \else%
    \setlength{\unitlength}{\svgwidth}%
  \fi%
  \global\let\svgwidth\undefined%
  \global\let\svgscale\undefined%
  \makeatother%
  \begin{picture}(1,0.44380261)%
    \put(0,0){\includegraphics[width=\unitlength]{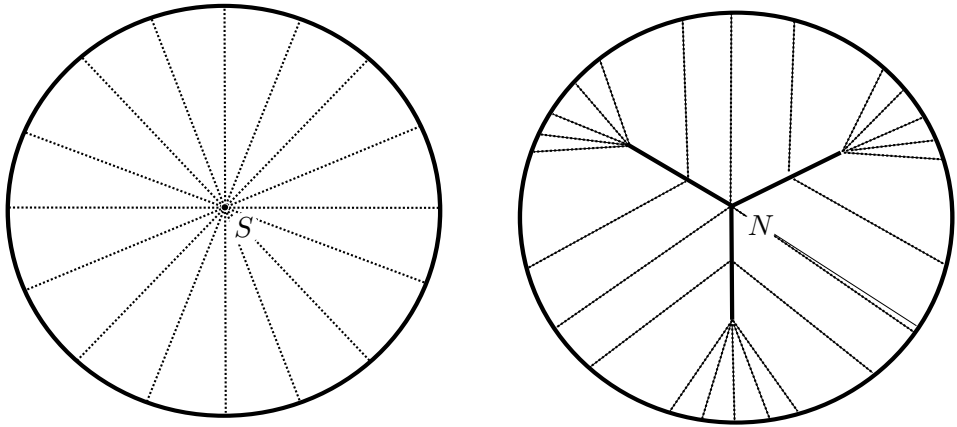}}%
    \put(0.7805629,0.20225655){\color[rgb]{0,0,0}\makebox(0,0)[lb]{$N$}}%
    \put(0.24054607,0.19990857){\color[rgb]{0,0,0}\makebox(0,0)[lb]{$S$}}%
  \end{picture}%
\endgroup%
\caption{Geodesic fields in the southern hemisphere (left) and the northern hemisphere (right).}
\label{F:Fields}
\end{figure}

%-------------------

We consider the fields  shown in Figure~\ref{F:Fields}. The left hand side corresponds to the southern hemisphere, equipped with the standard metric, and the right hand side corresponds to the northern hemisphere, where we consider a tree with three edges along great circles meeting at the north pole and making an angle of $2\pi/3$. 
The dotted lines denote the geodesic fields. We take the gluing diffeomorphism to be the identity. After the gluing, the tree is the cut locus of the south pole $S$.  By construction, this new metric contains no geodesic loops passing through the north pole $N$. 

For a general manifold $M$, take the above metric in $\mathbb{S}^n$ and connect it to $M$ by a narrow tube as in \cite[Section 8, Figure 6]{GS}. \end{proof}

The fact that geodesic loops converge to geodesic loops in the $C^\infty$ topology for the space of Riemannian metrics on a closed Riemannian manifold implies the following corollary to the preceding theorem.

% COR

\begin{cor}Let $M$ be a closed smooth manifold. Then there exists an open set  of Riemannian metrics  on $M$ (in the $C^\infty$ topology) such that not every point is the midpoint of a geodesic loop with the property that half the loop is minimal.
\end{cor}

%------------------------
% BIBLIOGRAPHY
%-------------------------


\begin{thebibliography}{10}

\bibitem{BIVZ} B\'{a}r\'{a}ny, I., Itoh, J., V\^{i}lcu, C., Zamfirescu, T., \emph{Every point is critical}, Adv. in Math. \textbf{235} (2013),  390--397.

\bibitem{BT} Bott, R., Tu, L.~W., \emph{Differential forms in algebraic topology}, Graduate Texts in Mathematics, 82. Springer-Verlag, New York-Berlin, 1982. 

\bibitem{Chee} Cheeger, J., \emph{Critical points of distance functions and applications to geometry}, Geometric topology: recent developments (Montecatini Terme, 1990), 1--38, Lecture Notes in Math., 1504, Springer, Berlin, 1991.

\bibitem{GS} Gluck, H., Singer, D., \emph{Scattering of geodesic fields. I}, Ann. of Math. (2) \textbf{108} (1978), no. 2, 347--372

\bibitem{Gro} Grove, K., \emph{Critical point theory for distance functions}, Differential geometry: Riemannian geometry (Los Angeles, CA, 1990), 357--385, Proc. Sympos. Pure Math., 54, Part 3, Amer. Math. Soc., Providence, RI, 1993. 

\bibitem{GroShio} Grove, K., Shiohama, K., \emph{A generalized sphere theorem}, Ann. of Math.  \textbf{106} (1977) no. 1 201--211.

\bibitem{Pet}Petersen, P. \emph{Riemannian geometry}, Second edition. Graduate Texts in Mathematics, 171. Springer, New York, 2006.  

\end{thebibliography}
\end{document}